\documentclass[11pt]{amsart}
\usepackage{xcolor} % (avant tikz)
\usepackage{amsthm,amsmath,graphicx,tikz}
\usetikzlibrary{arrows,calc,matrix}
\usepackage{etex}
\newtheorem{theorem}{Theorem}[section]
\newtheorem{proposition}[theorem]{Proposition} 
\newtheorem{lemma}[theorem]{Lemma}
\newtheorem{corollary}[theorem]{Corollary}
\newtheorem{conjecture}[theorem]{Conjecture}
\newcommand{\binomial}[2]{\left(\begin{array}{c} #1 \\ #2 \end{array}\right)}
\input xy
\xyoption{all}

\begin{document}

\author{Bruce Fontaine} % \and Pierre-Guy Plamondon}
\address{ Bruce Fontaine \\
          Department of Mathematics \\
          310 Malott Hall \\
          Cornell University \\
          Ithaca, NY USA.}
\email{bfontain@math.cornell.edu}

\author{Pierre-Guy Plamondon}
\address{ Pierre-Guy Plamondon \\
          Universit\'e de Paris Sud XI \\
UMR 8628 du CNRS \\
Laboratoire de Math\'ematiques - B\^at. 425 \\
91405 Orsay Cedex \\
France. }          
\email{pierre-guy.plamondon@math.u-psud.fr}
\title{Counting friezes in type $D_n$}

\keywords{Friezes, triangulations, punctured disk, Dynkin types, cluster algebras}

\begin{abstract}
We prove that there is a finite number of friezes in type $D_n$, and we provide a formula to count them.  As a corollary, we obtain formulas to count the number of friezes in types $B_n$, $C_n$ and $G_2$.  We conjecture finiteness (and precise numbers) for other Dynkin types.
\end{abstract}
\maketitle

%--------------------------------------------
\section{Introduction}

% I guess we should recall the standard coxeter Conway Frieze,
% define the idea of a generalized frieze as something like a
% mutation compatible evaluation of all cluster variables as
% integers. Then recall the unitary/non unitary divide and show
% the standard example of the non unitary D_4 frieze.

% I agree that this is where we should define friezes, refering
% to Coxeter-Conway, Caldero-Chapoton and Assem-Reutenauer-Smith.
% We should mention work of Morier-Genoud--Ovsienko--Tabachnikov;
% they proved the D_4 case and conjectured the E_6 case.  We should
% acknowledge the MSRI, since this is where everything started,
% and Dylan Thurston. (PG)

Friezes of type $A_n$ were defined by Coxeter \cite{Coxeter71} and studied by Conway and Coxeter \cite{ConCox73} in the early '70's.  An observation credited to Caldero in \cite{ARS10} is that Fomin and Zelevinsky's cluster algebras \cite{FZ02} allow for a huge generalization of the original definition. In this paper, we are interested in friezes of Dynkin types.

One way to define friezes is to say that they are ring homomorphisms from a cluster algebra to the ring of integers such that all cluster variables are sent to positive integers.  In Dynkin types, a cluster-free definition may be given as follows \cite[Section 3]{ARS10}. Let $C=(C_{i,j})_{n\times n}$ be a Cartan matrix of Dynkin type $\Delta$, and assume that we have an acyclic orientation of the associated Dynkin diagram.  Then a \emph{frieze of type $\Delta$} is a collection of positive integers $a(j,m)$, with $j\in \{1,\ldots,n\}$ and $m\in \mathbb{Z}$, such that
\begin{displaymath}
	a(j,m)a(j,m+1) = 1 + \Big(\prod_{j\to i}a(i,m)^{|C_{i,j}|}\Big)\Big(\prod_{i\to j}a(i,m+1)^{|C_{i,j}|}\Big).
\end{displaymath}
This is conveniently represented as in Figure \ref{figu::frieze}. For friezes of type $D_n$, there is a model developed by Schiffler \cite{Schiffler08} (see also \cite{BM09} and \cite{FST}) involving tagged arcs in a punctured polygon.  We recall this model in section \ref{sect::model}.
\begin{figure}[ht]
\begin{displaymath}
	\xymatrix@-1.5pc{& 1\ar[dr] && 3\ar[dr] && 2\ar[dr] && 2\ar[dr] && 3\ar[dr] && 1\ar[dr] && 3\ar[dr] && \cdots \\
	                 \cdots\ar[ur]\ar[dr]&& 2\ar[dr]\ar[ur] && 5\ar[dr]\ar[ur] && 3\ar[dr]\ar[ur] && 5\ar[dr]\ar[ur] && 2\ar[dr]\ar[ur] && 2\ar[dr]\ar[ur] && 5\ar[dr]\ar[ur] \\
								   \cdots\ar[r] & 3\ar[r]\ar[ur]\ar[dr] & 2\ar[r] & 3\ar[r]\ar[ur]\ar[dr] & 2\ar[r] & 7\ar[r]\ar[ur]\ar[dr] & 4\ar[r] & 7\ar[r]\ar[ur]\ar[dr] & 2\ar[r] & 3\ar[r]\ar[ur]\ar[dr] & 2\ar[r] & 3\ar[r]\ar[ur]\ar[dr] & 2\ar[r] & 3\ar[r]\ar[ur]\ar[dr] & 2\ar[r] & \cdots \\
									 \cdots\ar[ur] &&2\ar[ur] && 2\ar[ur] && 4\ar[ur] && 2\ar[ur] && 2\ar[ur] && 2\ar[ur] && 2\ar[ur] &
	}
\end{displaymath}
\caption{A frieze in type $D_5$.}
\end{figure}\label{figu::frieze}

Conway and Coxeter proved in \cite{ConCox73} that in type $A_n$, there is only a finite number of friezes, and that this number is the $(n+1)$-st Catalan number.  In type $D_4$, Morier-Genoud, Ovsienko and Tabachnikov \cite{MGOT12} proved that there are 51 friezes, a result conjectured by Propp \cite{Propp05} (in fact, they were working with $2$-friezes, which in these small cases are related to friezes).  In this paper, we extend these results to arbitrary $D_n$ types:

\begin{theorem}[\ref{theo::main}]
 The number of $D_n$ friezes is $\sum_{m=1}^n d(m)\binomial{2n-m-1}{n-m}$, where $d(m)$ is the number of divisors of $m$.
\end{theorem}

As a corollary to this and to the results in \cite{ConCox73}, we can count friezes in types $B_n$, $C_n$ and $G_2$ by folding Dynkin diagrams:
\begin{corollary}[\ref{theo::C}, \ref{theo::B}, \ref{theo::G}] The number of friezes in type $B_n$, $C_n$ and $G_2$ is $\sum_{m\leq\sqrt{n+1}}\binomial{2n-m^2+1}{n}$, $\binomial{2n}{n}$ and $9$, respectively. 
\end{corollary}

The theory of cluster algebras of Fomin and Zelevinsky \cite{FZ02} provides a way to construct friezes, namely by specializing variables of a given cluster to $1$.  Such friezes are called \emph{unitary friezes} in \cite{MG02}. All friezes of type $A_n$ are obtained in this way; however, Figure \ref{figu::frieze} provides an example of a frieze of type $D_5$ that does not arise in this fashion.  Thus it is worth noting that in types $B_n$, $D_n$ and $G_2$, the number of friezes is strictly greater than the number of clusters (given in \cite[Table 3]{FZ03}).  The sequences of numbers of friezes in types $D_n$ and $B_n$ make up two new entries in the On-Line Encyclopedia of Integer Sequences \cite{OEIS_Dn} \cite{OEIS_Bn}.

Note also that if $C$ is any Cartan matrix of non-Dynkin type, then it follows again from the theory of cluster algebras that there is an infinite number of friezes.

For the other Dynkin types, we propose the following
\begin{conjecture}[\ref{conj::other}]
 The number of friezes of type $E_6$, $E_7$, $E_8$ and $F_4$ is $868$, $4400$, $26952$ and $112$, respectively.
\end{conjecture}
Note that the number for type $E_6$ was conjectured already by Propp \cite{Propp05}, and evidence for this number was further obtained by Morier-Genoud, Ovsienko and Tabachnikov \cite{MGOT12}.

Finally, we would like to thank Dylan Thurston for some helpful conversations, MSRI for supporting us during the Cluster Algebras semester where this research began and the Sage mathematics software and community. We would also like to thank Dylan Rupel for his comments on an earlier version of the paper, and Michael Cuntz for pointing out a typo in one of our conjectures and sharing with us some of his computations.  We also thank an anonymous referee for his/her careful reading of the paper and numerous suggestions for improving it.

%--------------------------------------------
\section{Friezes of type $A_n$}
%........................
Let us begin by recalling the main results of Conway and Coxeter on type $A_n$ friezes. Of interest in \cite{ConCox73} were configurations of integers on a diamond grid such that the entries were strictly positive in a strip of height $n+2$, zero outside the strip and every $2$ by $2$ diamond with at least $3$ entries in the strip had determinant $1$. 

For this to occur, both the top and bottom row of the frieze had to be $1$ and as noted above, their main result is that if the height is $n+2$, then the number of such friezes is the $(n+1)$-st Catalan number. This result follows directly from a connection they establish to triangulations of the $(n+3)$-gon.

\begin{theorem}\label{theo::ConCox}\cite{ConCox73} The friezes of height $n+2$ are in correspondence with labellings of the diagonals of an $(n+3)$-gon with positive integers such that each quadrilateral satisfies the Ptolemy relation (see leftmost picture in Figure \ref{figu::relations}).  Moreover, in any frieze, the arcs with label $1$ form a triangulation of the polygon.
\end{theorem}

In the above theorem we consider the boundary arcs of the polygon to have label $1$ and the Ptolemy relation in this scenario is simply that the sum of the products of the labels of the opposite sides of a quadrilateral is the product of the labels of the diagonals. One can easily move back and forth between these two models in the following way: If one considers the middle $n$ entries of a zig-zag column in a frieze of height $n+2$ one can apply these as labels of the edges of a zig-zag triangulation of the $(n+3)$-gon.

Note that since the $A_n$ cluster algebra contains cluster variables corresponding to each diagonal arc of an $(n+3)$-gon (see \cite[Section 12.2]{FZ03bis}) and these variables are related by the Ptolemy relation, we immediately see that the Conway and Coxeter formulation of friezes is either an evalutation of the cluster variables so that each is a strictly positive integer or a ring map from the cluster algebra to the ring of integers so that each cluster variable maps to a strictly positive integer.

%--------------------------------------------
\section{Friezes of type $D_n$}
%........................
\subsection{Triangulations of the punctured polygon}\label{sect::model}
Recall that we defined a frieze as an evaluation of all cluster variables, where each variable is evaluated in a positive integer. We will now describe a geometric model due to Schiffler \cite{Schiffler08}.

Let $n\geq 4$ be an integer. Consider a \emph{once punctured $n$-gon} $\mathcal{P}_n$, that is, a connected orientable Riemann surface with one boundary component containing $n$ \emph{marked points} and one marked point in its interior, called the \emph{puncture}.  An \emph{arc} is an isotopy class of paths in $\mathcal{P}_n$ whose endpoints are marked points, which are not self-intersecting (except perhaps at the endpoints), whose interior are in the interior of $\mathcal{P}_n$ and which are not isotopic to a path contained on the boundary of $\mathcal{P}_n$.

A \emph{tagged arc} is an arc together with a possible ``notch'' at each of its endpoints.  We represent notches by  ``bowties'', see Figure \ref{figu::relations}.  The notches are asked to satisfy the following rules:
\begin{itemize}
 \item endpoints on the boundary are not notched;
 \item if both endpoints of an arc are the same marked point, then the endpoints are either both notched or both not notched.
\end{itemize}
We further require that tagged arcs do not cut out a once-punctured monogon.

Two tagged arcs $\alpha$ and $\beta$ are \emph{compatible} if one the following holds:
\begin{enumerate}
 \item $\alpha$ and $\beta$ are the same tagged arc, or
 \item at least one of $\alpha$ and $\beta$ has both endpoints on the boundary of $\mathcal{P}_n$, and the arcs can be represented in such a way that their interiors do not cross, or
 \item $\alpha$ and $\beta$ both have the puncture as an endpoint but their other endpoint differ, and they are both notched or both unnotched at the puncture, or
 \item $\alpha$ and $\beta$ have the same endpoints, one of which is the puncture, and exactly one of them is notched at the puncture.
\end{enumerate}

A maximal collection of pairwise compatible tagged arcs of $\mathcal{P}_n$ is a \emph{tagged triangulation}.  An example is given in Figure \ref{figu::triangulation}.  All tagged triangulations of $\mathcal{P}_n$ have exactly $n$ distinct arcs.

\begin{theorem}[\cite{Schiffler08}]The cluster variables of $D_n$ correspond to the (tagged) arcs in a once punctured $n$-gon. Moreover, the exchange relations are those of Figure \ref{figu::relations}.

\end{theorem}

\begin{center}
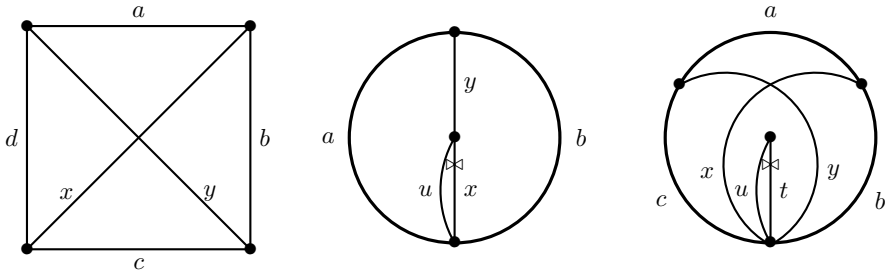
\begin{figure}[h]
\begin{tikzpicture}[scale=0.7]
\foreach \x in {1,2,...,4} {
\draw (90*\x+45:3) node {$\bullet$} ;
} ;
\draw[thick] (45:3) -- (135:3) node[midway, above, scale=0.85] {$a$} ;
\draw[thick] (135:3) -- (-135:3) node[midway, left, scale=0.85] {$d$} ;
\draw[thick] (-135:3) -- (-45:3) node[midway, below, scale=0.85] {$c$} ;
\draw[thick] (-45:3) -- (45:3) node[midway, right, scale=0.85] {$b$} ;
\draw[thick] (45:3) -- (-135:3) node[near end, left, scale=0.85] {$x$} ;
\draw[thick] (135:3) -- (-45:3) node[near end, right, scale=0.85] {$y$} ;

\begin{scope}[xshift=6cm]
\draw[very thick] (0,0) circle (2) ;
\draw (0,2) node {$\bullet$} ;
\draw (0,-2) node {$\bullet$} ;
\draw (-2.4,0) node[scale=0.85] {$a$} ;
\draw (2.4,0) node[scale=0.85] {$b$} ;
\draw (0,0) node {$\bullet$} ;

\draw[thick] (0,0) -- (0,2) node[midway, right, scale=0.85] {$y$} ;
\draw[thick] (0,-2) -- (0,0) node[midway, right, scale=0.85] {$x$} node[near end, sloped, rotate=90, scale=0.75] {$\bowtie$} ;
\draw[thick] (0,0) arc (150:210:2);
\draw (-0.55,-1) node[scale=0.85] {$u$} ;

\end{scope}

\begin{scope}[xshift=12cm]
\draw[very thick] (0,0) circle (2) ;
\draw (30:2) node {$\bullet$} ;
\draw (150:2) node {$\bullet$} ;
\draw (-90:2) node {$\bullet$} ;
\draw (90:2.4) node[scale=0.85] {$a$} ;
\draw (-30:2.4) node[scale=0.85] {$b$} ;
\draw (-150:2.4) node[scale=0.85] {$c$} ;
\draw (0,0) node {$\bullet$} ;

\draw[thick] (150:2) arc (120:-60:1.75);
\draw (-1.2,-0.66) node[scale=0.85] {$x$} ;
\draw[thick] (30:2) arc (60:240:1.75);
\draw (1.2,-0.66) node[scale=0.85] {$y$} ;
\draw[thick] (0,-2) -- (0,0) node[midway, right, scale=0.85] {$t$} node[near end, sloped, rotate=90, scale=0.75] {$\bowtie$} ;
\draw[thick] (0,0) arc (150:210:2) ;
\draw (-0.55,-1) node[scale=0.85] {$u$} ;

\end{scope}

\end{tikzpicture}
\caption{Ptolemy relation $xy=ac+bd$ (left) and other relations $xy = a+b$  (middle) and $xy=bc+atu$ (right). The middle relation also holds if $u$ and $y$ are tagged and $x$ is untagged, and the left one does for any tagging.}
\label{figu::relations}
\end{figure}
\end{center}

Thus, as is noted in \cite{BM09}, a $D_n$ frieze is simply a choice of positive integer weight for each (tagged) arc in the punctured disk model, satisfying the relations of Figure \ref{figu::relations} (and where boundary arcs are always assumed to have weight $1$).  In the rest of the paper, this is the point of view from which we will view friezes.

%.......................
\subsection{Description of all friezes}
We will prove the following proposition, which describes the friezes in type $D_n$ and ensures that there is only a finite number of them.

\begin{proposition}\label{prop::description}
 From any frieze of type $D_n$ can be extracted a unique tagged triangulation $T$ of the punctured $n$-gon in such a way that
 \begin{enumerate}
  \item $T$ contains all arcs of weight $1$ which are not notched;
  \item all arcs of $T$ connecting marked points on the boundary have weight $1$;
  \item either $T$ has only two arcs incident with the puncture, both having the same endpoints, different notchings and weight $1$, or the $m$ arcs of $T$ incident with the puncture are not notched and all have the same weight, which can be any divisor of $m$.
 \end{enumerate} 
 In particular, there is only a finite number of friezes of type $D_n$.
\end{proposition}

% Insert example of such a triangulation in a, say, 8-gon.
\begin{center}
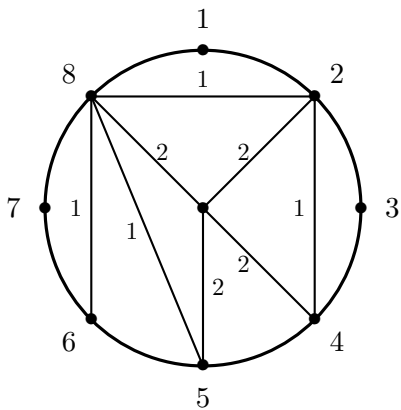
\begin{figure}%\label{figu::triangulation}
\begin{tikzpicture}[scale=0.7]
\draw[very thick] (0,0) circle (3) ;
\foreach \x in {1,2,...,8} {
\draw (-45*\x:3) node {$\bullet$} ;
\draw (-45*\x+135 :3.6) node {\x} ;
} ;
\draw (0,0) node {$\bullet$} ;
\draw[thick] (135:3) -- (45:3) node[midway, above, scale=0.85] {1} ;
\draw[thick] (-45:3) -- (45:3) node[midway, left, scale=0.85] {1} ;
\draw[thick] (135:3) -- (-90:3) node[midway, left, scale=0.85] {1} ;
\draw[thick] (135:3) -- (-135:3) node[midway, left, scale=0.85] {1} ;
\draw[thick] (135:3) -- (0,0) node[midway, right, scale=0.85] {2} ;
\draw[thick] (45:3) -- (0,0) node[midway, left, scale=0.85] {2} ;
\draw[thick] (-45:3) -- (0,0) node[midway, left, scale=0.85] {2} ;
\draw[thick] (-90:3) -- (0,0) node[midway, right, scale=0.85] {2} ;

%\draw[thick] (60:3) -- (0,0) node[midway, right, scale=0.85] {$\tau a$} 
%node[near end, sloped, rotate=90, scale=0.75] {$\bowtie$} ;

\end{tikzpicture}
\caption{An example of a triangulation as in Proposition \ref{prop::description}.}
\label{figu::triangulation}
\end{figure}
\end{center}

Figure \ref{figu::triangulation} gives an example of a triangulation satisfying (1), (2) and (3). If such a triangulation exists for a given frieze, then its uniqueness is clear.  We prove its existence in several steps.  First, we show that there is indeed a triangulation containing all the arcs of weight $1$ of the frieze:

\begin{lemma}
 Two arcs of weight $1$ in a frieze of type $D_n$ cannot be incompatible.
\end{lemma}
\begin{proof}{ When two arcs are incompatible, then their weights, say $x$ and $y$, have to satisfy a relation of the form $xy = m_1 + m_2$, where $m_1$ and $m_2$ are monomials in the weights of the other arcs (see Figure \ref{figu::relations}).  In particular, $xy$ is at least $2$, so $x$ and $y$ cannot both be equal to $1$.
}
\end{proof}

Let $T_0$ be the triangulation of the punctured $n$-gon defined in the three following steps.  

\emph{Step 1.} Include in $T_0$ all arcs of weight $1$ that are not incident to the puncture.  Any such arc cuts $\mathcal{P}_n$ into a smaller punctured polygon an a smaller unpunctured polygon.  The latter defines a smaller frieze of type $A$, and by the results of \cite{ConCox73} (see Theorem \ref{theo::ConCox}), it contains a unique triangulation of arcs with weight $1$.  Thus, by adding all these arcs in $T_0$, we obtain a smaller punctured disc to which are glued unpunctured discs, each with a triangulation of $1$'s.

\emph{Step 2.} Add in $T_0$ all arcs (notched or not) from the boundary to the puncture which have weight $1$.  By the following Lemma, proved by Hugh Thomas in an appendix to \cite[Proposition A.2]{BM09}, either we add no arcs in Step 2, or we add enough arcs to make $T_0$ into a triangulation:

\begin{lemma}
 If one of the arcs of a frieze of type $D_n$ incident with the puncture has weight $1$, then the frieze contains a triangulation of arcs of weight $1$.  In particular, if one of the arcs of $T_0$ incident with the puncture has weight $1$, then all arcs of $T_0$ have weight $1$.
\end{lemma}
\begin{proof}{Assume that an arc of a frieze of type $D_n$ incident with the puncture has weight $1$.  Then all the arcs compatible with this one form a frieze of type $A_{n-1}$; in particular, by \cite{ConCox73}, there is a triangulation consisting of arcs of weight $1$.
}
\end{proof}

\emph{Step 3} Add all un-notched arcs compatible with the ones already there and connecting marked points on the boundary to the puncture.

Then $T_0$ is a well-defined triangulation which satisfies conditions (1) and (2) of Proposition \ref{prop::description}. Moreover, the arcs added in step 1 cut $T_0$ into smaller unpunctured triangulated discs and one ``central'' punctured triangulated disc.  If we cut away all unpunctured disc, we are left with a triangulation of a punctured disc $\mathcal{P}_m$ (with $m\geq 2$, since arcs with their endpoints on the boundary never cut out a monogon) on which all arcs having their endpoints on the boundary have weight $\geq 2$.  Thus we are in the situation of the following Lemma:

\begin{center}
\begin{figure}
\begin{tikzpicture}[scale=0.7]
\draw[very thick] (0,0) circle (3) ;
\foreach \x in {1,2,...,5} {
\draw (-72*\x:3) node {$\bullet$} ;
\draw (-72*\x+135 :3.6) node {\x} ;
} ;
\draw (0,0) node {$\bullet$} ;
\draw[thick] (288:3) -- (72:3) node[midway, right, scale=0.85] {1} ;
\draw[thick] (288:3) -- (0,0) node[midway, right, scale=0.85] {2} ;
\draw[thick] (72:3) -- (0,0) node[midway, left, scale=0.85] {2} ;
\draw[thick] (144:3) -- (0,0) node[midway, left, scale=0.85] {2} ;
\draw[thick] (216:3) -- (0,0) node[midway, right, scale=0.85] {2} ;
\end{tikzpicture}
\caption{The triangulation $T_0$ for the $D_5$ frieze in Figure~\ref{figu::frieze}. The arcs with weight $7$ are the arcs of length $4$ from vertices $1$ to $2$ and from $2$ to $3$.}
\end{figure}
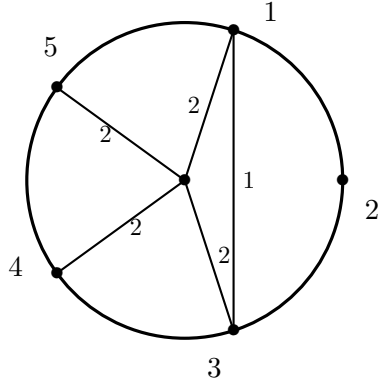
\end{center}

\begin{lemma}\label{lemm::no1}
 Assume that a frieze on a punctured disc $\mathcal{P}_m$ (with $m\geq 2$) is such that all arcs not incident to the puncture have weight greater than $1$.  If $m\geq 3$, then
 \begin{enumerate}
  \item All un-notched arcs incident with the puncture have the same weight.  The same is true for notched arcs.
  \item Any arc not incident with the puncture and forming a $(d+1)$-gon not containing the puncture has weight $d$.
  \item If $t$ and $u$ are the weights of the notched and un-notched arcs, respectively, incident with the puncture, then $tu=m$.
 \end{enumerate}
 If $m=2$, then there are two compatible arcs having weight $1$.
\end{lemma}
\begin{proof}{The case $m=2$ has to be treated separately, so assume first that $m\geq 3$. Let $a_1, a_2, \ldots, a_m$ be the weights of the un-notched arcs incident with the puncture, in clockwise order. Without loss of generality, we can assume that $a_1 \geq a_i$ for all $i\in \{1,\ldots,m\}$.  Let $t_i$ be the weight of the arc forming a triangle with the arcs weighted $a_i$ and $a_{i+2}$, where the indices are viewed modulo $m$.  Figure \ref{figu::weights} illustrates this in an octogon.
\begin{center}
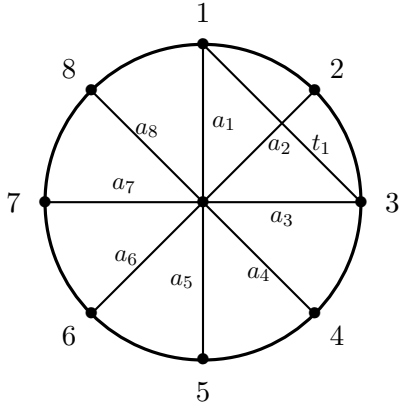
\begin{figure}[h]
\begin{tikzpicture}[scale=0.7]
\draw[very thick] (0,0) circle (3) ;
\foreach \x in {1,2,...,8} {
\draw (-45*\x:3) node {$\bullet$} ;
\draw (-45*\x+135 :3.6) node {\x} ;
} ;
\draw (0,0) node {$\bullet$} ;
\draw[thick] (-45*1+135:3) -- (0,0) node[midway, right, scale=0.85] {$a_1$} ;
\draw[thick] (-45*2+135:3) -- (0,0) node[midway, right, scale=0.85] {$a_2$} ;
\draw[thick] (-45*3+135:3) -- (0,0) node[midway, below, scale=0.85] {$a_3$} ;
\draw[thick] (-45*4+135:3) -- (0,0) node[midway, below, scale=0.85] {$a_4$} ;
\draw[thick] (-45*5+135:3) -- (0,0) node[midway, left, scale=0.85] {$a_5$} ;
\draw[thick] (-45*6+135:3) -- (0,0) node[midway, left, scale=0.85] {$a_6$} ;
\draw[thick] (-45*7+135:3) -- (0,0) node[midway, above, scale=0.85] {$a_7$} ;
\draw[thick] (-45*8+135:3) -- (0,0) node[midway, above, scale=0.85] {$a_8$} ;
\draw[thick] (-45*1+135:3) -- (-45*3+135:3) node[near end, above, scale=0.85] {$t_1$} ;

\end{tikzpicture}
\caption{Labelling of the weights.}
\label{figu::weights}
\end{figure}
\end{center}
For each $i$, there is a Ptolemy relation $a_{i+1}t_i = a_i+a_{i+2}$.  By our assumptions, $t_i\geq 2$, and by maximality of $a_1$, we get
\begin{displaymath}
	2a_1 \leq a_1 t_m = a_m+a_2 \leq 2a_1,
\end{displaymath}
so $2a_1 = a_m+a_2$, which implies that $a_m=a_1=a_2$, again by maximality of $a_1$.  This argument propagates around the polygon, so by induction, we get that all the $a_i$'s are equal.  This proves part (1) for un-notched arcs; the proof for notched arcs is the same.

By the above relations, $a_{i+1}t_i = a_i + a_{i+2} = 2a_{i+1}$, so $t_i=2$ for all $i$.  Thus all arcs forming a triangle with the boundary have weight $2$.  We prove (2) by induction from here: assume that for a given $d$, all arcs forming a $(d+1)$-gon with the boundary have weight $d$.  Let $z$ be the weight of an arc forming a $(d+2)$-gon with the boundary.  Then there is a Ptolemy relation of the form $a_{i+d}z = a_{i} + a_{i+d+1}d$, so $a_{i+d}z = a_{i+d}(1+d)$, and therefore $z=d+1$.  Part (2) is proved.

Part (3) follows from part (2) and from the relation on the right in Figure \ref{figu::relations}.

The case $m=2$ (\emph{i.e} $\mathcal{P}_m$ is a digon) is proved by noticing that triangulations of the punctured digon are associated to cluster algebras (or friezes) of type $A_1\times A_1$.  There are only four arcs, they all touch the puncture, they have weight $1$ or $2$, and exactly two of them have weight $1$ and are compatible.
}
\end{proof}

Now, construct the triangulation $T$ as follows: if $T_0$ has exactly two arcs touching the puncture, both having the same boundary endpoint (and thus different notchings), then take $T=T_0$.  Else, if the arcs added in step 2 were notched, replace them by their un-notched version to get $T$.  Then $T$ still satisfies condition (1) and (2), and it follows from Lemma \ref{lemm::no1} that $T$ satisfies condition (3).  Indeed, cutting along all arcs of weight $1$, we are left with a smaller punctured polygon whose arcs not incident to the puncture have weight at least $2$ and form a frieze of type $D$.  Thus Lemma \ref{lemm::no1} applies.  This finishes the proof of Proposition \ref{prop::description}.

%In other words, friezes in type $D_m$ without any $1$ all look like
%\begin{displaymath}
% \xymatrix@-1.5pc{ 2\ar[dr] && 2\ar[dr] && 2\ar[dr] && \cdots \\
%            & 3\ar[ur]\ar[dr] && 3\ar[ur]\ar[dr] && 3\ar[ur]\ar[dr] && \cdots \\
%            && 4\ar[ur]\ar[dr] && 4\ar[ur]\ar[dr] && 4\ar[ur]\ar[dr] && \cdots \\
%            &&& \cdots\ar[ur]\ar[dr] && \cdots\ar[ur]\ar[dr] && \cdots\ar[ur]\ar[dr] && \cdots \\
%            &&&& m-1\ar[ur]\ar[r]\ar[dr] & d\ar[r] & m-1\ar[ur]\ar[r]\ar[dr] & m/d \ar[r] & m-1\ar[ur]\ar[r]\ar[dr] & d\ar[r] & \cdots \\
%            &&&&& m/d\ar[ur] && d\ar[ur] && m/d\ar[ur]
% }
%\end{displaymath}
%with $d$ a divisor of $m$.
%
%\emph{Proof of Proposition \ref{prop::description}.} 

%.......................
\subsection{Triangulations of punctured $n$-gons}\label{sect::triangulations}

In this section, we replace all notched arcs connecting a boundary marked point $M$ to the puncture by the corresponding arc joining $M$ to $M$ and cutting out a punctured monogon.

 Let $T_{n,m}$ be the number of triangulations of a once-punctured $n$-gon with exactly $m$ un-notched arcs, or spokes, from the outer marked points to the inner puncture.
\begin{theorem}\label{theo::countingTriang}
$T_{n,m}=\binomial{2n-m-1}{n-1}$.
\end{theorem}

\begin{lemma} We have $T_{n,m}=\frac{n}{m}\sum_{i_1+\cdots+i_m=n-m} \prod_j C_{i_j}$, where $C_n$ is the $n$-th Catalan number.
\end{lemma}

\begin{proof} Given a triangulation of the punctured $n$-gon with $m$ spokes, the portion of the triangulation between two adjacent spokes is an honest triangulation of a $(k+2)$-gon, where $k$ is the number of vertices contained in between the two spokes. The two extra vertices are the end points of the spokes themselves. Thus there are $C_k$ possible triangulations that fit between the two given spokes. The total number of vertices not involved with the spokes is $n-m$, so we partition $n-m$ into $m$ non-negative pieces $i_1+\cdots+i_m=n-m$ with $i_j\geq 0$. Fix one of the spokes as a starting point, then we should see $\sum_{i_1+\cdots+i_m=n-m}\prod_j C_{i_j}$ triangulations. This under counts the true number since rotating a triangulation one step can give a different triangulation. Thus if we multiply by $n$, the total number of possible rotations, we would count each triangulation at least once. But we are ignoring the fact that we fixed one of the $m$ spokes, so we are now over counting by a factor of 
$m$.
 This leaves us with $T_{n,m}=\frac{n}{m}\sum_{i_1+\cdots+i_m=n-m} \prod_j C_{i_j}$.
\end{proof}

We can now prove Theorem \ref{theo::countingTriang}:
\begin{proof}
Recall that $c(x)=\frac{1-\sqrt{1-4x}}{2x}$ is the generating function for the Catalan numbers. The coefficient of $x^n$ in $(c(x))^k$ is known as the ballot number $B(n,k)$ and has closed form $B(n,k)=\frac{k}{2n+k}\binomial{2n+k}{n}$. But, $T_{n,m}=\frac{n}{m}B(n-m,m)=\binomial{2n-m-1}{n-1}$.
%need a reference here for Ballot numbers? This is essentially due to Catalan directly? 1887 sur les nombres de Segner. Rend Circ Mat Pal 1, pp190 201  --- Neat reference, I've never cited such and old paper :-) (PG)
\end{proof}

Since the generating function for the $k$ ballot numbers is $(c(x))^k$, then the sum $1+(c(x))y+(c(x))^2y^2+\cdots=\frac{1}{1-yc(x)}$ is a two variable generating function for the ballot numbers. If we examine $\frac{1}{1-xyc(x)}$, then we see that the coefficient for $x^ny^m$ is $B(n-m,m)$. 

\begin{proposition} The generating function for $T_{n,m}$ is $\frac{1}{(c(x)-2)(1-xyc(x))}$.
\end{proposition}

\begin{proof}
Note that $\frac{1}{1-xyc(x)}$ is almost a generating function for $T_{n,m}$, it is off by a factor of $\frac{n}{m}$ in term $x^ny^m$. This can be corrected by integration and differentiation: $$\int\left(\frac{x}{y}\frac{d}{dx}\left(\frac{1}{1-xyc(x)}\right)\right)dy=\frac{c(x)+xc'(x)}{c(x)(xyc(x)-1)}.$$ One can check that $1+\frac{xc'(x)}{c(x)}=\frac{1}{2-c(x)}$, in which case the generating function becomes $\frac{1}{(2-c(x))(1-xyc(x))}$.
\end{proof}

%.......................
\subsection{Counting friezes}
We can now prove our main theorem.

\begin{theorem}\label{theo::main}
The number of $D_n$ friezes is $\sum_{m=1}^n d(m)\binomial{2n-m-1}{n-m}$.
\end{theorem}

\begin{proof}

A frieze is determined by its weights on a single cluster, or (tagged) triangulation.  This, together with Proposition \ref{prop::description}, tells us that the number of friezes is $\sum_{m=1}^n d(m)T_{n,m}$, where $T_{n,m}$ is as in section \ref{sect::triangulations}.  The result follows from Theorem \ref{theo::countingTriang}.

%We will show that any frieze has a particularly nice (tagged) triangulation where all arcs between exterior vertices have weight $1$ and the remaning $m$ 'spokes' into the puncture are (non tagged?) arcs all have the same weight, a divisor of $m$.

%Given a frieze, consider the set of all arcs between exterior vertices of weight $1$. As noted in the finiteness proof, this is a non-crossing set. More over, if we consider an arc of weight $1$ joining the vertices $i$ and $j$, it splits the punctured $n$-gon into a punctured $k$-gon and an $n-k$-gon. By the work of Coxter-Conway ref{} the $n-k$ gon is a frieze of $A_{n-k-3}$ type, and thus there is a triangulation of it where all arcs are weight $1$. This shows that if one also includes a maximal (non crossing) set of spokes along with the arcs of weight $1$ between exterior vertices, we have a triangulation. More over, if the maximal set of spokes is of size $m$, then then the finiteness proof shows that we have a frieze if and only iff all of the $m$ spokes are labelled with a divisor of $m$

%Since the number of such triangulation is $T_{n,m}$ is follows that the number of friezes is $\sum_{m=1}^n d(m)T_{n,m}$.

%--I've transformed this argument into a separate proposition. (PG).
\end{proof}

%--------------------------------------------
\section{Friezes of other Dynkin types}

In order to deal with non-ADE type quivers, one must switch to the world of cluster algebras defined by a skew-symmetrizable exchange matrix $(b_{ij})$. In such a matrix, there exists strictly positive integers $d_i$ such that $d_ib_{ij}=-d_jb_{ji}$. From this data, it is possible to create a valued quiver: If $b_{ij}$ is strictly positive, we add an arrow from vertex $j$ to $i$ labelled $(b_{ij},-b_{ji})$. This places at most one arrow between any pair of vertices since either both $b_{ij}$ and $b_{ji}$ are $0$ or exactly one is negative. Once can thus move back and forth between valued quivers (and the data $d_i$) and skew-symmetrizable matrices.

Following the folding method of \cite{Dupont08}, given a quiver $\Delta$ coming from a simply-laced Dynkin diagram and a group $G$ of automorphisms, one can obtain the quotient quiver $\Delta/G$ by identifying the vertices that lie in the same orbit. We also identify the edges that lie in the same orbit. The resulting edges are given the label $(1,1)$, except in the special case that multiple edges are identified that share an endpoint. If $i$ such edges are identified, then if the resulting edge is directed away from the vertex, it is given the label $(i,1)$; otherwise, it is given the label $(1,i)$.

In particular, if one were to apply the construction to the standard $D_4$ quiver with edges oriented outwards from the central vertex and take the quotient via the order $3$ automorphism, the resulting graph has $2$ vertices with one edge labelled $(1,3)$, and its exchange matrix corresponds to the Cartan matrix for $G_2$.

To summarize, if $\Delta$ is a simply laced Dynkin quiver and $G$ a group of automorphisms, then $\Delta/G$ is a valued quiver. Dupont concludes that the action of $G$ lifts to the cluster algebra $A(\Delta)$, thus by \cite[Corollary 5.16]{Dupont08}, $A(\Delta/G)$ can be identified with a subalgebra of $A(\Delta)/G$. Moreover, \cite[Theorem 7.3]{Dupont08} gives equality since $\Delta$ is Dynkin. The projection $\pi:A(\Delta)\to A(\Delta)/G$ can then be thought of as a surjective ring homomorphism from $A(\Delta)$ to $A(\Delta/G)$, which sends the cluster variables of $A(\Delta)$ to the cluster variables of $A(\Delta/G)$ via a quotient by $G$.

\begin{lemma} Let $\Delta$ by a Dynkin quiver and $G$ a group of automorphisms, then each $\Delta/G$ frieze gives rise to a $\Delta$ frieze. Moreover, each $\Delta$ frieze that is $G$-invariant descends to a $\Delta/G$ frieze.
\end{lemma}

\begin{proof}
For the first part, if we consider a $\Delta/G$ frieze to be a ring homomorphism from the cluster algebra $A(\Delta/G)$ to $\mathbb{Z}$, then composing with the map $\pi$ gives a $\Delta$ frieze.

For the second part, a $\Delta$ frieze that is $G$-invariant descends to a ring homomorphism from $A(\Delta)/G$ to $\mathbb{Z}$ and thus gives a $\Delta/G$ frieze under the identification of $A(\Delta)/G$ with $A(\Delta/G)$.
\end{proof}

For the case of $B_n$, $C_n$ and $G_2$, these are quotients of $D_{n+1}$, $A_{2n-1}$ and $D_{4}$ respectively where the automorphisms we use are the maps swapping the short arms of $D_{n+1}$, mirroring $A_{2n-1}$ through the middle vertex and the order $3$ rotation of $D_4$.

\begin{theorem}\label{theo::C} The number of $C_n$ friezes is $\binomial{2n}{n}$.
\end{theorem}

\begin{proof}
Since $C_n$ is a folding of $A_{2n-1}$, by the above lemma, each $C_n$ frieze can be lifted to a unique $A_{2n-1}$ frieze which is $G$-invariant. One can check that the action of $G$ on the $A_{2n-1}$ cluster variables is given by the following action on the arcs of a $2n+2$-gon: take an arc and map it to the arc whose end points are diametrically opposed to the originals. Recall from \cite{ConCox73} that the set of arcs in the $2n+2$-gon that are labeled $1$ must form a triangulation. But the image of each arc labeled $1$ under $G$ is also an arc labeled $1$, so the triangulation is $G$-invariant. Thus we have a $G$-invariant cluster in $A_{2n-1}$ on which the frieze evaluates to $1$, but by \cite{Dupont08}, this descends to a cluster of $C_n$. 

Thus each $C_n$ frieze is determined by fixing one cluster with every variable being $1$ and the number of $C_n$ friezes is the number of $C_n$ clusters, $\binomial{2n}{n}$ (see \cite[Table 3]{FZ03}).
\end{proof}

\begin{theorem}\label{theo::B} The number of $B_n$ friezes is $\sum_{1\leq m\leq\sqrt{n+1}}\binomial{2n-m^2+1}{n}$.
\end{theorem}

\begin{proof}
Since $B_n$ is a folding of $D_{n+1}$, each $B_n$ frieze lifts to $D_{n+1}$ frieze which is $G$-invariant. The two nodes on the end of $D_{n+1}$ which are identified by $G$ correspond to an untagged/tagged pair of parallel arcs in the punctured $n+1$-gon. Thus it follows that the label assigned to each pair is the same. Now as outlined in the calculation of the $D_{n+1}$ friezes, when we decompose a frieze into a partial triangulation of all arcs labeled $1$, and a $D_m$ frieze containing no 1's, the $D_m$ contains at least one spoke from the $D_{n+1}$ frieze. Moreover, in Lemma \ref{lemm::no1}, we see that the product of an untagged spoke with its parallel tagged spoke in the $D_m$ frieze is $m$. Thus $m$ must be a perfect square and moreover, the only $D_m$ frieze which is allowed is a frieze with the square root labeling the spokes. Applying this reduction to the $D_{n+1}$ formula results in the given formula.
\end{proof}

\begin{theorem}\label{theo::G} The number of $G_2$ friezes is $9$.
\end{theorem}

\begin{proof}
Since $G_2$ is a folding of $D_4$, each $G_2$ frieze lifts to a $D_4$ frieze which is $G$-invariant. Of the $50$ $D_4$ friezes which come from setting a cluster to all $1$'s, only $8$ are $G$-invariant and thus correspond to the $8$ $G_2$ friezes which also come from setting a cluster to all $1$'s. The remaining frieze assigns $2$ to the outer nodes of $D_4$ and $3$ to the center, and this is also $G$-invariant, so it descends to the sole remaining $G_2$ frieze, leaving us with $9$ friezes.
\end{proof}

What remains are the sporadic $E_6$, $E_7$, $E_8$ and $F_4$. 

\begin{conjecture}\label{conj::other}
The number of $E_6$, $E_7$ and $E_8$ frieze are $868$, $4400$ and $26952$ respectively. Since $F_4$ is a folding of $E_6$, the number of $F_4$ friezes is 112.
\end{conjecture}

These numbers are obtained by computer computation but depend on the next conjecture. It should be noted that a certain subset of friezes is easy to obtain, namely the unitary friezes (see \cite{MG02}), which arise from setting all cluster variables in a single cluster to $1$.

\begin{conjecture}
The value of a frieze at a node in a Dynkin diagram is less than the maximal value of the node over the set of unitary friezes.
\end{conjecture}

Since the set of unitary friezes is computable (i.e. using Sage for instance), this puts an easily computed maximal bound on the entries in a frieze. This conjecture is true for type $A_n$ and $C_n$ since the only friezes there are unitary. The case of $B_n$ and $G_2$ would follow from the case of $D_n$ and we prove this case below:

\begin{theorem} Given $F$, a $D_n$ frieze, it has an ideal triangulation $T$, as determined above. Let k be the number of spokes in $T$. Let $T_1$ be the frieze whose ideal triangulation is $T$ with the spokes in $T$ labelled by $k$ and $T'_1$ the same but with spokes labelled by $1$. Suppose that for an arc $a$, $F(a)$ is the label of $F$ at $a$. Then if $a$ is untagged we have $F(a)\leq T_1(a)$ and if $a$ is tagged then $F(a)\leq T'_1(a)$.
\end{theorem}

Note that both $T_1$ and $T'_1$ are unitary friezes. In the case of $T'_1$ this is clear, and in the case of $T_1$ it is sufficient to note that if one swaps the spokes of $T$ for the corresponding tagged spokes, all arcs are now labelled $1$.

The above theorem is a direct consequence of the following two lemmas:

\begin{lemma}
Consider the $D_n$ cluster algebra and set all cluster variables corresponding to the spokes in $T$ to a new variable x. Given any untagged spoke in $D_n$, the Laurent polynomial in terms of the cluster $T$ factors as $xf$ where $f$ is a Laurent polynomial in the cluster variables from the non-spoke arcs in $T$.
\end{lemma}

\begin{proof}
Take any untagged spoke in $D_n$, if this spoke is in $T$, then we are done. Otherwise let $A$ be the vertex at which the spoke ends. It lives between two spokes in $T$, ending at vertices $B$ and $C$. The arc $BC$ in $T$ and hence is a cluster variable $y$. By $BA$ and $CA$, we mean the arcs between the given vertices that lie between the two spokes ending at $B$ and $C$. Let $b$ and $c$ be their Laurent polynomials in terms of the cluster for $T$. We note that $b$ and $c$ involve no cluster variables that correspond to spokes in $T$. In particular they only involve non-spoke arcs lieing between the spokes $B$ and $C$. The Laurent polynomial for the spoke ending at $A$, after setting all spokes in $T$ to $x$ is $(xb+xc)/y=x(b+c)/y$ where $(b+c)/y$ is Laurent in the non-spoke variables in $T$.
\end{proof}

We can also swap the words tagged and untagged in the above statement and proof. This proves the above theorem for all tagged and untagged spokes, since the maximum and minimum value of $x$ that can occur in a frieze are $1$ and $k$ and the portion $f$ doesn't depend on the the value of $x$.

\begin{lemma}
Consider an arc in $D_n$ between two boundary vertices and its Laurent polynomial in terms of the cluster $T$. If we set all cluster variables corresponding to spokes in $T$ to $x$ then the resulting Laurent polynomial does not depend on x.
\end{lemma}

\begin{proof}
Let $B$ and $C$ be the end points of the arc. We then induct on the the (minimal) number of spokes in $T$ the arc crosses. If it is $0$, then the arc is either already in $T$ or is contained in a region of $T$ bounded by non spoke arcs and hence can be obtained by mutating within this region. If it is $1$, then then let $A$ be the end point of the single spoke in $T$ that the arc crosses. By $BA$ and $CA$ we denote the arcs between the respective vertices that do not cross the given arc between $B$ and $C$, i.e. so that the triangle $ABC$ does not contain the puncture. Let $b$ and $c$ be the Laurent polynomials for $BA$ and $CA$ after setting all spoke variables in $T$ to $x$. Since $BA$ and $CA$ cross no spokes, $b$ and $c$ are independent of $x$. Let $i$, $j$ and $k$ be the the Laurent polynomials for the spokes ending at $A$, $B$ and $C$ after setting all spoke variables in $T$ to $x$. By the above lemma $i=x*\tilde{i}$, $j=x*\tilde{j}$ and $k=x*\tilde{k}$ where $\tilde{i}$, $\tilde{j}$ and $\tilde{k}$ are independent of x. In fact $\tilde{i}=1$ since $A$ is the end point of a spoke in $T$. But then the Laurent polynomial for $BC$ is $$\frac{x\tilde{j}c+x\tilde{k}b}{x\tilde{i}} = \frac{\tilde{j}c+\tilde{k}b}{\tilde{i}}$$ which no longer depends on x. 

Now suppose that the arc crosses $m$ spokes in $T$. Pick a boundary vertex $A$ that lies (not necessarly strictly) between the spokes that the arc crosses. Pick arcs $BA$ and $AC$ as above, they then cross strictly fewer spokes in $T$, so by induction do not have Laurent polynomials depending on $x$. But the spokes to $A$, $B$ and $C$ do, as in the above calculation shows that the Laurent polynomial for the arc $BC$ does not depend on $x$.
\end{proof}

This shows that we have $F(a)=T_1(a)=T'_1(a)$ for all arcs between boundary vertices and the theorem is proved.

The listing of friezes and the programs used to generate them are available at \cite{PlamondonWebpage}.

%------------------------------------------------


\begin{thebibliography}{99}

\bibitem{ARS10} I. Assem, C. Reutenauer and D. Smith, \emph{Friezes}, Adv. Math. {\bf 225} (2010), 3134-3165.
\bibitem{BM09} K. Baur and R. Marsh, \emph{Frieze patterns for punctured discs}, J. Algebr. Comb. (2009) 30: 349-379.
\bibitem{ConCox73} J. H. Conway and H. S. M. Coxeter, \emph{Triangulated polygons and frieze patterns}, Math. Gaz. {\bf 57} (1973), 87-94 and 175-183.
\bibitem{Coxeter71} H. S. M. Coxeter, \emph{Frieze patterns}, Acta Arith. {\bf 18} (1971), 297-310.
\bibitem{Dupont08} G. Dupont, \emph{An approach to non-simply-laced cluster algebras}, J. Algebra {\bf 320} (2008), 1626-1661.
\bibitem{FST} S. Fomin, M. Shapiro, D. Thurston, \emph{Cluster algebras and triangulated surfaces. Part I: Cluster complexes}, Acta Mathematica {\bf 201} (2008), 83-146.
\bibitem{FZ02}  S. Fomin and A. Zelevinsky, \emph{Cluster algebras. I. Foundations}, J. Amer. Math. Soc. {\bf 15} (2002), 497-529.
\bibitem{FZ03bis}  S. Fomin and A. Zelevinsky, \emph{Cluster algebras. II. Finite type classification},  Invent. Math. 154 (2003), no. 1, 63-121. 
\bibitem{FZ03} S. Fomin and A. Zelevinsky, \emph{$Y$-systems and generalized associahedra}, Annals of Mathematics {\bf 158} (2003), 977-1018.
\bibitem{MG02} S. Morier-Genoud, \emph{Arithmetics of $2$-friezes}, J. Algebraic Combin. {\bf 36(4)} (2012), 519-539.
\bibitem{MGOT12} S. Morier-Genoud, V. Ovsienko and S. Tabachnikov, \emph{2-frieze patterns and the cluster structure of the space of polygons}, Ann. Inst. Fourier {\bf 62}, 3 (2012) 937-987. 
\bibitem{Propp05} J. Propp, \emph{The combinatorics of frieze patterns and Markoff numbers}, arXiv:math/0511633 [math.CO].
\bibitem{Schiffler08} R. Schiffler, \emph{A geometric model for cluster categories of type $D_n$}, J. Alg. Comb. {\bf 29}, no. 1, (2008) 1-21.
\bibitem{OEIS_Dn} \begin{verbatim}oeis.org/A247415\end{verbatim}
\bibitem{OEIS_Bn} \begin{verbatim}oeis.org/A247416\end{verbatim}
\bibitem{PlamondonWebpage} \begin{verbatim}http://www.math.u-psud.fr/­~plamondon/friezes/\end{verbatim}


\end{thebibliography}
\end{document}